\documentclass[12pt,russian,a4paper]{article}

\usepackage{cmap}
\usepackage[utf8]{inputenc}
\usepackage[T2A]{fontenc}

\usepackage[russian]{babel}

\usepackage{hyperref}

\usepackage{amssymb}
\usepackage{amsmath}
\usepackage{amsthm}

\newcommand{\setK}{\mathcal{K}}
\newcommand{\setJ}{\mathcal{J}}
\newcommand{\setI}{\mathcal{I}}

\newcommand{\setR}{\mathbb{R}}

\newcommand{\walls}[1]{\left | #1 \right |} 
\newcommand{\pars}[1]{\left( #1 \right)} 
\newcommand{\class}[1]{[ #1 ]} 
\newcommand{\braces}[1]{\left\{ #1 \right\}} 

\newcommand{\condset}[2]{\braces{\, #1 \mid #2 \,}} 

\newcommand{\eps}{\epsilon}

\newcommand{\mysum}{\sum\limits}
\newcommand{\mymax}{\max\limits}
\newcommand{\myproduct}{\prod\limits}

\renewcommand{\emptyset}{\varnothing}
\renewcommand{\epsilon}{\varepsilon}
\renewcommand{\phi}{\varphi}
\renewcommand{\le}{\leqslant}
\renewcommand{\ge}{\geqslant}

\newtheorem{theorem}{Теорема}
\newtheorem*{theorem-star}{Теорема}

\newtheorem*{statement-star}{Утверждение}

\newtheorem{hypothesis}{Гипотеза}
\newtheorem*{hypothesis-star}{Гипотеза}

\newtheorem{corollary}{Следствие}
\newtheorem*{corollary-star}{Следствие}

\newtheorem{lemma}{Лемма}
\newtheorem*{lemma-star}{Лемма}

\newtheorem{observation}{Наблюдение}
\newtheorem*{observation-star}{Наблюдение}

\newtheorem*{designation-star}{Обозначение}

\newtheorem*{designations}{Обозначения}

\newtheorem*{proposition-star}{Предложение}

\newtheorem*{property-star}{Свойство}

\theoremstyle{remark}

\theoremstyle{definition}
\newtheorem{definition}{Определение}
\newtheorem*{definition-star}{Определение}

\theoremstyle{definition}

\newtheorem*{example-star}{Пример}

\theoremstyle{definition}

\begin{document}
\author{А.\,С. Дмитриев, А.\,Б. Дайняк}
\title{Единственность экстремального графа в задаче о максимальном количестве независимых множеств в регулярных графах}
\maketitle
\begin{abstract}
    Основная цель настоящего препринта — изложить доказательство того, что при $2k | n$ существует единственный (с точностью до изоморфизма) граф, на котором достигается максимум количества независимых множеств (НМ) среди всех $k$-регулярных графов на $n$ вершинах, а также дать верхнюю оценку на количество НМ в графах, «сильно отличающихся» от этого экстремального. (См. Теорему 5 и Следствие 2.) Все излагаемые здесь результаты также изложены в выпускной квалификационной работе бакалавра, успешно защищённой первым автором под руководством второго автора 30 июня 2015 г. на кафедре \href{http://www.discrete-mathematics.org}{дискретной математики} ФИВТ \href{http://mipt.ru}{МФТИ}.
    Частично излагаемые результаты перекрываются с недавней работой~\cite{Davies2015}.
\end{abstract}

\section{Введение}

В данной работе рассматриваются только простые графы --- неориентированные графы без петель и кратных рёбер. С определениями базовых понятий теории графов, которые мы не приводим в настоящей работе, можно ознакомиться, например, по книге \cite{EmelBook}.

Независимым множеством называется подмножество вершин графа, никакие две из которых не смежны. Подсчёт и оценивание количества независимых множеств в графах из различных классов — задачи, занимающие заметное место в теории графов. К такого рода задачам сводятся некоторые задачи из математической химии \cite{Merrifield,Hosoya}, алгебры~\cite{Alon}. Кроме того, упомянутые задачи представляют самостоятельный теоретический интерес. 

Излагаемые в настоящей работе результаты относятся к задаче о верхней оценке количества независимых множеств в регулярных графах. Интерес именно к классу регулярных графов здесь обусловлен тем, что в алгебраических приложениях оценки количества независимых множеств применяются к графам типа Кэли, а те, в свою очередь, регулярны либо близки к таковым.

Н. Алоном в работе~\cite{Alon}, заложившей основу экстремальным задачам о количестве независимых множеств в регулярных графах, была выдвинута гипотеза, о том, что максимальное количество независимых множеств среди всех $d$-регулярных графов на $n$ вершинах достигается на двудольном графе, являющемся объединением $\frac{n}{2d}$ копий графа $K_{d,d}$. Такой граф мы будем называть \emph{графом Алона}.

\begin{hypothesis}[{\cite{Alon}}]\label{Hyp}
	Для любого $d$-регулярного графа $G$ на $n$ вершинах справедлива оценка:
	\[i(G) \le (2^{d + 1} - 1) ^ \frac{n}{2 d}.\]
\end{hypothesis}

Сам Алон доказал только более слабую оценку:
\begin{theorem}[{\cite[Theorem 1.1]{Alon}}]
	Для любого регулярного графа $G$ на $n$ вершинах:
	\[i(G) \le 2^{\frac{n}{2}\pars{1 + O \pars{{d^{-0.1}}}}}.\]
\end{theorem}

Эта оценка претерпела на протяжении ряда лет несколько улучшений (см., напр., \cite{Sapo,Galvin}). В конце концов, гипотеза~\ref{Hyp} была доказана как комбинация двух результатов: в 2001 г. Дж.~Кан~\cite{Kahn} доказал гипотезу для двудольных графов а в 2009 г. Ю.~Жао~\cite{Zhao} свел общую задачу о регулярных графах к задаче о \emph{двудольных} регулярных графах. Таким образом, просуществав почти 20 лет, гипотеза Алона была доказана. Однако открытым оставался вопрос о том, \emph{на каких в точности} графах достигается максимум количества независимых множеств; является ли граф Алона уникальным экстремальным. Решение этой задачи и является основной целью настоящей работы (см. Следствие~\ref{coroUniqueness}). Попутно нами рассмотрены вопросы о том, \emph{насколько сильно} падает количество независимых множеств в графе при удалении от графа Алона (Теорема \ref{DmitrievNotBipartite} дает оценку на число независимых множеств в <<сильно недвудольных>> графах, а теорема \ref{DmitrievAny} дает оценку для графов, компоненты связности которых отличаются от $K_{d,d}$).

\begin{designations}	
Для произвольного графа $G$ мы вводим следующие обозначения:
	\begin{itemize}
	\item $V(G)$ --- множество вершин графа $G$;
	\item $E(G)$ --- множество рёбер графа $G$;
	\item $\setI(G)$ --- семейство независимых множеств графа $G$;
	\item $i(G) = \walls{\setI(G)}$;
	\item $G[A]$ --- подграф в $G$, порождённый множеством вершин~$A$.
	\end{itemize}
\end{designations}

\begin{designation-star}
Пусть $G$ --- граф, $v$ --- вершина $G$.

Обозначим $N_G(v)$ множество вершин $u$, таких, что $uv$ --- ребро в~$G$. Вершины $N_G(v)$ будем называть \emph{соседями} вершины $v$.
\end{designation-star}

Всюду в работе, если основание логарифма не указано явно, оно полагается равным $2$.

\section{Недвудольные графы}
Цель данного раздела — доказательство теоремы~\ref{DmitrievNotBipartite}, оценивающей количество независимых множеств в графах, <<сильно отличных от двудольных>>. Мотивация к получению такой оценки вызвана, в частности, особой ролью двудольных графов в доказательстве гипотезы Алона~\cite{Kahn}.

\begin{definition}
	Назовем множество вершин $A \subset V(G)$ независимым с множеством вершин $B \subset V(G)$ в графе $G$, если в $G$ нет ребра вида $ab$, где $a \in A, b \in B$.
\end{definition}

\begin{designations}
	$\setK(G) = \left\{ A \subset V(G) \mid G[A] \text{ --- двудольный} \right\}$  \\
	Обозначим $\setJ(G)$ множество пар $(A, B)$ таких, что $A \subset V(G), B \subset V(G)$, $A$~независимо с $B$ и $A \cup B \in \setK(G)$.
\end{designations}

\begin{definition}
	Назовем \emph{размером} пары подмножеств $(A, B) \in V(G) \times V(G)$ сумму их мощностей $\walls{A} + \walls{B}$.
\end{definition}

\begin{lemma}[{\cite[Lemma 2.1]{Zhao}}]
	Для любого графа $G$ существует биекция между $\setI(G) \times \setI(G)$ и $\setJ(G)$, сохраняющая размер любой пары подмножеств.
\end{lemma}

Введём функцию $P(\lambda, G)$, называемую \emph{многочленом независимости} (см., напр., \cite{Zhao}):

\begin{designation-star}
    \[ P(\lambda, G) = 	\mysum_{I \in \setI(G) } {\lambda^{\walls{I}}}. \]
\end{designation-star}

При $0 \le \lambda \le 1$, величина $P(\lambda, G)$ есть матожидание количества независимых множеств в графе $G[X]$, где $X$ содержит каждую из вершин $G$ независимо с вероятностью $\lambda$. \\
В частности, легко видеть, что $P(1, G) = i(G)$.

\begin{theorem}\label{DmitrievNotBipartite}
	Пусть $G$ --- $d$-регулярный граф на $n$ вершинах, $C$ --- количество недвудольных порожденных подграфов в $G$, $\lambda \ge 1$. Тогда $P(\lambda, G)^2 \le P(\lambda, K_{d,d})^{\frac{n}{d}} - 2 \cdot C$.
\end{theorem}

\begin{proof}
    Аналогично доказательству Жао из \cite[Theorem 1.2]{Zhao}, имеем: \\
    
	$
	P(\lambda, G \times K_2) = 
	\mysum_{I \in \setI(G \times K_2) } {\lambda^{\walls{I}}} = 
	\mysum_{\substack{A, B \in V(G) \\ A \text{ незав. с } B}} {\lambda^{\walls{A} + \walls{B}}} = \\ = 
	\mysum_{\pars{A, B} \in \setJ(G)}{\lambda^{\walls{A} + \walls{B}}} 
	+ \mysum_{\substack{A, B \in V(G) \\ A \text{ незав. с } B \\ A \cup B \notin \setK(G)}}{\lambda^{\walls{A} + \walls{B}}} = 
	\mysum_{(A, B) \in \setI(G) \times \setI(G)} \lambda^{\walls{A} + \walls{B}}
	+ T = 
	P(\lambda, G) ^ 2 + T,
	$
	
	где через $T$ обозначено $\mysum_{\substack{A, B \in V(G) \\ A \text{ незав. с } B \\ A \cup B \notin \setK(G)}}{\lambda^{\walls{A} + \walls{B}}}$.

	Оценим теперь $T$:

	$
	T = 
	\mysum_{\substack{A, B \in V(G) \\ A \text{ незав. с } B \\ A \cup B \notin \setK(G)}}{\lambda^{\walls{A} + \walls{B}}} \ge
	\mysum_{\substack{A, B \in V(G) \\ A \text{ незав. с } B \\ A \cup B \notin \setK(G)}}1 \ge
	\mysum_{\substack{A, B \in V(G) \\ A \text{ незав. с } B \\ A \cup B \notin \setK(G)\\B = \emptyset}}1 
	+
	\mysum_{\substack{A, B \in V(G) \\ A \text{ незав. с } B \\ A \cup B \notin \setK(G)\\A = \emptyset}}1 
	= \\ =
	\mysum_{\substack{A \in V(G) \\ A \notin \setK(G)\\}}1 
	+
	\mysum_{\substack{B \in V(G) \\ B \notin \setK(G)\\}}1 
	= 2 \cdot C.
	$

	В последней цепочке неравенств второе неравенство верно, так как если $(A, B) = (\emptyset, \emptyset)$, то $A \cup B = \emptyset \in \setK(G)$.
	
	Таким образом,
	$ P(\lambda, G) ^ 2 = P(\lambda, G \times K_2) - T \le P(\lambda, G \times K_2) - 2 \cdot C$.
	
	Учитывая, что $G \times K_2$ двудолен, состоит из $2 \cdot n$ вершин и $d$-регулярен: \\
	$P(\lambda, G) ^ 2 \le P(\lambda, K_{d, d})^{\frac{2n}{2d}} - 2 \cdot C$.
\end{proof}

\begin{corollary}
	Пусть $G$ --- $d$-регулярный граф на $n$ вершинах, $C$ --- количество недвудольных индуцированных подграфов в  $G$. Тогда $i(G)^2 \le i(K_{d,d})^{\frac{n}{d}} - 2 \cdot C$.
\end{corollary}
\begin{proof}
	Воспользуемся теоремой~\ref{DmitrievNotBipartite}, положив $\lambda = 1$ и тем, что для любого графа $G$ выполнено равенство $i(G) = P(1, G)$.
\end{proof}

\begin{corollary}
	Пусть $G$ --- недвудольный регулярный граф на $n$ вершинах. Тогда $i(G) < i(K_{d, d})^{\frac{n}{2d}}$.
\end{corollary}
\begin{proof}
	$G$ недвудольный, значит его подграф индуцированный $V(G)$ недвудолен, а тогда
	$i(G)^2 \le i(K_{d,d})^{\frac{n}{d}} - 2 \cdot 1 < i(K_{d,d})^{\frac{n}{d}}$,
	откуда получаем необходимое утверждение.
\end{proof}

\section{Двудольные графы}

В данном разделе изучается количество независимых множеств в регулярных двудольных графах. Главной целью данного раздела является доказательство следующей теоремы:

\begin{theorem}\label{DmitrievBipartite}
    Для любого целого положительного $d$ существует такая константа $D$ (зависящая только от $d$), что для любого $d$-регулярного связного двудольного графа $G$ на $n$ вершинах, не изоморфного $K_{d,d}$ выполнено:
    \[i(G) \le D ^ \frac{n}{2 d},\]
    при этом $D < 2^{d + 1} - 1$.
\end{theorem}

\begin{lemma}\label{LemmaNoKDD}
    Пусть $G$ --- $d$-регулярный двудольный граф на $n$ вершинах. Пусть $v$ --- вершина $G$, лежащая в компоненте связности графа $G$ отличной от $K_{d, d}$. Тогда существуют вершины $u$ и $w$, такие, что $u, w \in N_G(v)$ и $N_G(u) \ne N_G(w)$.
\end{lemma}

\begin{proof}
    Обозначим $u$ --- произвольного соседа вершины $v$.
    
    Предположим, что для любых двух соседей вершины $v$ их множества соседей совпадают. Рассмотрим любую пару вершин $a \in N_G(v)$, $b \in N_G(u)$. Заметим, что они соединены, так как $N_G(u) = N_G(a)$ по предположению, а значит $b \in N_G(a)$.
    
    Заметим, что для каждой из вершин из $N_G(v)$ и $N_G(u)$ мы показали наличие $d$ ребер, выходящих из неё. Так как $G$ $d$-регулярен, они не имеют других ребер. Следовательно,  $N_G(v) \cup N_G(u)$ образует компоненту связности графа $G$ изоморфную $K_{d,d}$. Противоречие. Предположение не верно и лемма доказана.
    
\end{proof}

\begin{observation}\label{DmitrievMax}
    Пусть $f, g: [a, b] \to \setR $ --- две числовые функции, причем $f < g$ на $C \subset {[a, b]}$ и $\mymax_{[a, b]} f = \mymax_C f$. Тогда $\mymax_{[a,b]} f <\mymax_{[a,b]} g$.
\end{observation}

\begin{proof}
    Пусть $x$ --- точка, в которой достигается максимум $f$ на $C$. Тогда:
    \[\mymax_{[a, b]} f = \mymax_C f = f(x) < g(x) \le  \mymax_{[a,b]} g.\]
\end{proof}

Для доказательства теоремы~\ref{DmitrievBipartite} нам потребуется понятие энтропии. Приведем здесь определения и несколько лемм, которые нам будут нужны. Подробнее ознакомиться с этим понятием (в том числе, ознакомиться с доказательствами лемм) можно, например, в~\cite{EntropyBook}.

\begin{definition-star}
    Пусть $p_1, p_2, \ldots, p_n$ --- распределение вероятностей, то есть $ \forall i: 0 \le p_i \le 1$ и $\mysum_{i = 1}^n p_i = 1$. Тогда энтропией $p_1, p_2, \ldots, p_n$ называется
    \[H(p_1, p_2, \ldots, p_n) = -\mysum_{i=1}^n p_i \log p_i, \]
    при этом, если $p_i = 0$, то слагаемое $p_i \log p_i$ полагается равным $0$.
\end{definition-star}

\begin{definition-star}
    Пусть $X$ --- случайная величина принимающая значения $x_1, x_2, \ldots, x_n$ с вероятностями $p_1, p_2, \ldots, p_n$ и только их. Тогда энтропией $H$ называется 
    \[H[X] = H(p_1, p_2, \ldots, p_n). \]
    Заметим, что добавление значений с нулевой вероятностью не меняет значения энтропии.
\end{definition-star}

\begin{definition-star}
    Пусть $X$ случайная величина, $\mathcal{A}$ --- событие. Тогда $H[X|\mathcal{A}]$ --- энтропия случайной величины с условным распределением $P_{X|\mathcal{A}}$.
\end{definition-star}

\begin{definition-star}
    Пусть $X, Y$ --- случайные величины, принимающие значения $x_1, x_2, \ldots, x_n$ и $y_1, y_2, \ldots, y_m$ соответственно. Тогда условной энтропией $X$ при условии $Y$ называется
    \[ H[X|Y] = -\mysum_{j=1}^m P(Y = y_j) H[X|Y = y_j] \]
\end{definition-star}
\begin{lemma}\label{EntropyLog}
    Пусть $p_1, p_2, \ldots, p_n$ --- распределение вероятностей. Тогда \[H(p_1, p_2, \ldots, p_n) \le \log n,\] при этом равенство достигается тогда и только тогда, когда~ $\forall i: p_i = \frac{1}{n}$.
\end{lemma}

\begin{lemma}\label{CondEntropy}
    \[H[(X, Y)] = H[X] + H[Y \mid X]. \]
\end{lemma}

\begin{lemma}\label{EntropyVector}
    Энтропия векторной величины не превосходит суммарной энтропии её компонент. Другими словами,
    
    \[H(X_1, X_2, \ldots, X_n) \le \mysum_{i = 1}^n H[X_i].\]
\end{lemma}

\begin{lemma}[{\cite[Lemma 1]{Radha}}]\label{EntropyWierd}
     Пусть $(X_1, X_2, \ldots, X_n)$ --- векторная случайная величина.
     
     Для множества $S \subset \left\{1, 2, \ldots n \right\}$ обозначим $X_S = \condset{X_i}{i \in S}$.
     
     Пусть множества $S_1, S_2, \ldots S_m$ таковы, что каждое число из $\left\{1, 2, \ldots n \right\}$ входит в хотя бы $k$ из этих множеств. Тогда
    \[ \mysum_{i = 1}^m H[X_{S_i}] \ge k H[X]. \]
\end{lemma}

\begin{lemma}\label{EntropyDepend}
    Пусть $X, Y$ --- дискретные случайные величины, $f$ --- произвольная функция и $Y=f(X)$.
    Тогда:
    \begin{itemize}
    \item $H[Y \mid X] = 0$;
    \item $H[(X, Y)] = H[X]$;
    \item Для любой случайной величины $Z$ выполнено: $H[Z \mid X] \le H[Z \mid Y] $.
    \end{itemize}
\end{lemma}

\begin{lemma}\label{DmitrievEntropy}
    Для любых $n \ge 2, \eps > 0$ существует число $D_{n, \eps} < \log n$, такое что для любой случайной величины $X$, принимающей $n \ge 2$ различных значений с вероятностями $p_1, p_2, \ldots, p_n$, такими что $p_1 - p_2 \ge \eps > 0$ выполнено:
    \[
     H[X] \le D_{n, \eps}.
    \]
\end{lemma}
\begin{proof}
    Так как $p_1 - p_2 \ge \eps$, выполнено хотя бы одно из следующих условий:
    \begin{itemize}
    \item $p_1 \ge \frac{1}{n} + \frac{\eps}{2}$;
    \item $p_2 \le \frac{1}{n} - \frac{\eps}{2}$.
    \end{itemize}
    Положим $D_{n, \eps} = \max \pars{
        \mymax_{0 \le q_i \le 1, q_1 \ge\frac{1}{n} + \frac{\eps}{2}} H(q_1, q_2, \ldots, q_n),
        \mymax_{0 \le q_i \le 1, q_2 \le\frac{1}{n} - \frac{\eps}{2}} H(q_1, q_2, \ldots, q_n)
    }$.
    
    Оба внутренних максимума существуют, как максимумы непрерывной функции $H$ на компакте. Оба из них меньше, чем $\log n$ по лемме~\ref{EntropyLog}.
\end{proof}

\begin{proof}[Доказательство теоремы ~\ref{DmitrievBipartite}]
    Доказательство повторяет доказательство Кана из \cite{Kahn}.
    
    Обозначим через $v_1, v_2, \ldots, v_n$ вершины $G$. Обозначим $L$, $R$ --- множества вершин каждой из долей.
  
    Пусть независимое множество $X$ выбирается равновероятно из $\setI(G)$. Будем отождествлять $X$ с его индикаторным вектором, то есть вектором $(X_1, X_2, \ldots, X_n)$, где $X_j = I\braces{v_j \in G}$.
    
    По лемме~\ref{EntropyLog}, $H[X] = \log i(G)$, с другой стороны, по лемме~\ref{CondEntropy}, $H[X] = H[(X_L, X_R)] = H[X_R] + H[X_L \mid X_R]$.
    
    Рассмотрим множества $N(v), v \in L$. Каждая вершина правой доли входит ровно в $d$ из этих множеств. Тогда по лемме~\ref{EntropyWierd} \[\mysum_{v \in L} H[X_{N(V)}] \ge H[X_R].\]
    
    По лемме~\ref{EntropyVector} имеем $H[X_L \mid X_R ] \le \mysum_{v \in L} H [ X_v \mid X_R ]$. \\
    По лемме~\ref{EntropyDepend},$H[X_v \mid X_R] \le H[X_v \mid N_G(v)]$, так как $N_G(v)$ однозначно определяется по $X_R$.
   
    Таким образом, нами уже получено:
    \[ 
    \log i(G) = H[x] \le \frac{1}{d} \mysum_{v \in L}H[X_{N_G(v)}] + \mysum_{v \in L} H[X_v \mid X_R] \le 
    \mysum_{v \in L} \pars{\frac{1}{d}H[X_{N_G(v)}] + H[X_v \mid N_G(v)]}
    .\]

    Введем для каждой вершины $v$ случайную величину $Y_v$, где $Y_v = 1$, если ни один сосед вершины $v$ не лежит в $X$, $Y_v = 0$ иначе. Обозначим также $p_v = P(Y_v = 1)$. Тогда если $X_{N_G(v)} \ne \vec{0}$, то $v$ не лежит в $X$(так как есть её соседи в $X$), а значит, используя лемму~\ref{EntropyLog} и факт, что $X_v$ принимает только 2 значения, получаем:
    
    \[
    \begin{array}{rl}
    H[X_v \mid X_{N_G(v)}] & = P(X_{N_G(v)} = \vec{0}) H[X_v \mid X_{N_G(v)} = \vec{0}] + P(X_{N_G(v)} \ne \vec{0}) H[X_v \mid X_{N_G(v)} \ne \vec{0}] =
    \\ & = p_v H[X_v \mid Y_v=1] \le p_v \log 2 = p_v .
    \end{array}\]
    
    Поскольку $Y_v$ это функция от $X_{N_G(v)}$:
    \[
    \begin{array}{rl}
    H[X_{N_G(v)}] = &H[(X_{N_G(v)}, Y_v)] = H[Y_v] + H[X_{N_G(v)} \mid Y_v] =\\= &H(p_v, 1 - p_v) + (1 - p_v) H[X_{N_G(v)} \mid X_{N_G(v)} \ne \vec{0}] + p_v [X_{N_G(v)} \mid X_{N_G(v)}=\vec{0}] = \\ = &H(p_v, 1 - p_v) + (1 - p_v) H[X_{N_G(v)} \mid X_{N_G(v)} \ne \vec{0}]
    .\end{array}\]
    
    Оценим теперь $H[X_{N_G(v)} \mid X_{N_G(v)} \ne \vec{0}]$. Рассмотрим вершины $u$, $w$, полученные из леммы~\ref{LemmaNoKDD} для вершины $v$. Их множества соседей не совпадают, поэтому существует такая вершина $t$, что она соединена с $w$, но не с $u$. Хотим оценить разницу вероятностей $P(X_{N_G(v) = \braces{u}}) - P(X_{N_G(v) = \braces{u, w}})$.
    Каждому независимому множеству $S$ такому, что $S \cap N_G(v) = \braces{u, w}$ сопоставим независимое множество $S \setminus \braces{w}$ для которого верно $S \setminus \braces{w} \cap N_G(v) = \braces{u}$, при этом разным $S$ сопоставлены разные независимые множества. Однако, существуют независимые множества $T$, для которых $T \cap N_G(v) = \braces{u}$, но при этом нет множества $S$, которому они сопоставлены. Назовём такие множества $T$ <<плохими>>.
    
    Оценим долю <<плохих>> множеств. Обозначим $U = N_G(v) \cup N_G(u) \cup N_G(t)$. Рассмотрим произвольное независимое множество $S$ в $G[V(G) \setminus U]$. По нему можно построить плохое независимое множество $S \cup \braces{u, t}$. Оно является независимым множеством так как $S \cap (N_G(u) \cup N_G(t)) = \emptyset$ и является плохим, так как содержит $u, t$ и $S \cap N_G(v) = \braces{u}$ в силу того, что $S \cap N_G(v) = \emptyset$. Легко заметить, что всего существует не более, чем $2^{\class{U}} \le 2^{3d}$ независимых множеств в $G$ таких, что их пересечение с $V(G) \setminus U$ равно $S$  . Таким образом мы получили, что при фикисрованном $S$ доля <<плохих>> множеств не превосходит $2^{-3d}$, а значит и общая доля <<плохих>> множеств от всех независимых множеств в $G$ не превосходит $2^{-3d}$.
    
    Рассмотрим $H[X_{N_G(v)} \mid X_{N_G(v)} \ne \vec{0}]$. Ясно, что $X_{N_G(v)}$ при указанном условии принимает не более $2^d - 1$ различных значений, при этом $P(X_{N_G(v) = \braces{u}}) - P(X_{N_G(v) = \braces{u, w}}) \ge 2^{-3d}$. Таким образом по лемме~\ref{DmitrievEntropy}, $H[X_{N_G(v)} \mid X_{N_G(v)} \ne \vec{0}] \le D_{2^d-1, 2^{-3d}} =: D_1$, где константа зависит только от $d$.
    
    Продолжим цепочку неравенств:
    \[
    \begin{array}{rl}
     \log i(G) &  \le \mysum_{v \in L} \pars{\frac{1}{d}H[X_{N_G(v)}] + H[X_v \mid N_G(v)]} \\
     & \le \mysum_{v \in L} \pars{\frac{1}{d}\pars{H(p_v, 1 - p_v) + (1 - p_v)D_1} + p_v} \\
     & \le \frac{n}{2} \mymax_{0 \le p \le 1}  \pars{\frac{1}{d}\pars{H(p, 1 - p) + (1 - p)D_1} + p} \\
     
     & \le \frac{n}{2d} \mymax_{0 \le p \le 1}  \pars{H(p, 1 - p) + (1 - p)  D_1 + pd}
    .\end{array} 
    \]
    Обозначим $f(p) = H(p, 1 - p) + (1 - p)  D_1 + pd, D_2 = \mymax_{0 \le p \le 1} f(p)$.
    
    Заметим, что $f(p) < H(p, 1 - p) + (1 - p)  D_1 + pd$ на $[0; 1)$ так как $D_1 < \log \pars{2^{d}-1}$. \\
    Заметим, что $f(1) = d < D_2$ так как из обратного следовало бы, что $i(G) \le 2^{\frac{n}{2}}$, что, очевидно, неверно.
    Тогда, согласно наблюдению~\ref{DmitrievMax}: 
    
    \[D_2 < \mymax_{0 \le p \le 1}  \pars{H(p, 1 - p) + (1 - p)   \log \pars{2^d-1}+ pd}.\]
    
    Согласно \cite{Kahn}, значение данного максимума равно $\log (2^{d + 1} - 1)$.
    
    Таким образом, положив $D = D_2$ получаем условие теоремы.
    
\end{proof}

Заметим, что условие двудольности на самом деле можно отбросить:
\begin{corollary}\label{DmitrievAny}
    Для любого целого положительного $d$ существует такая константа $D$ (зависящая только от $d$), что для любого $d$-регулярного связного графа $G$ на $n$ вершинах, не изоморфного $K_{d,d}$ выполнено:
    \[i(G) \le D ^ \frac{n}{2 d},\]
    при этом $D < 2^{d + 1} - 1$.
\end{corollary}
\begin{proof}
    Зафиксируем $d$. Обозначим через $D$ константу, полученную из теоремы~\ref{DmitrievBipartite}.
    
    Рассмотрим произвольный $d$-регулярный связный граф $G$, не изоморфный $K_{d, d}$. Возможны три случая:
    \begin{itemize}
    \item $G$ --- двудольный. Тогда $i(G) \le D ^ \frac{n}{2 d}$ по теореме~\ref{DmitrievBipartite}.
    \item $G$ --- недвудольный. Тогда из теоремы~\ref{DmitrievNotBipartite} получаем, что $i(G) ^ 2 < i(G \times K_2)$. При этом, $G \times K_2$ двудолен. $G$ недвудолен, следовательно существует вершина $v$ в $G$ и путь из $v$ в $v$ нечетной длины. Это значит, что между вершинами $v_1$ и $v_2$ графа $G \times K_2$, соответствующих вершине $v$ графа $G$ существует путь нечетной длины. Таким образом, вершины $v_1$ и $v_2$ лежат в одной компоненте связности, в разных долях графа $G \times K_2$ и при этом не соединены ребром, а значит $G \times K_2$ не является графом Алона. Воспользовавшись теоремой~\ref{DmitrievBipartite}, получаем:
    
    \[ i(G) < \sqrt{i(G \times K_2)} < \sqrt{D^\frac{2n}{2d}} = D^\frac{n}{2d}.\]
    \end{itemize}
    В обоих случаях получили необходимое утверждение.
\end{proof}

\begin{corollary}\label{coroUniqueness}
    Для любого целого положительного $d$ существует такая константа $D < 1$ (зависящая только от $d$), что для любого $d$-регулярного графа $G$ на $n$ вершинах, не являющегося графом Алона выполнено:
    \[i(G) \le D (2^{d + 1} - 1)^ \frac{n}{2 d}.\]
\end{corollary}
\begin{proof}
    Фиксируем $d$. Для него, согласно следствию~\ref{DmitrievAny}, существует константа $D_1$, меньшая $2^{d + 1} - 1$ такая, что для любого двудольного регулярного связного графа, не изоморфоного $K_{d, d}$ выполнено $i(G) \le D_1 ^ \frac{n}{2 d}$.
    
    Рассмотрим произвольный $d$-регулярный граф $G$, неизоморфный графу Алона. Хотя бы одна из его компонент связности не изоморфна $K_{d, d}$. Пусть компоненты связности $G_1, G_2, \ldots, G_t$ графа $G$ имеют $n_1, n_2, \ldots, n_t$ вершин, $\sum_{j=1}^t n_j = n$, причем $G_1$ не изоморфна $K_{d, d}$.
    
    Тогда:
    \[
    i(G) = \myproduct_{j=1}^t i(G_i) = i(G_1) \myproduct_{j=2}^t i(G_i) \le D_1^{\frac{n_1}{2d}} \myproduct_{j=2}^t  (2^{d + 1} - 1)^\frac{n_i}{2d}= \pars{\frac{D_1}{2^{d + 1} - 1}}^{\frac{n_1}{2d}} \myproduct_{j=1}^t  (2^{d + 1} - 1)^\frac{n_i}{2d} 
    \]
    
    Положив $D := \sqrt{\frac{D_1}{2^{d + 1} - 1}} < 1$ и заметив, что $n_1 \ge d$ получаем:
    \[i(G) \le D  \pars{2^{d + 1} - 1}^\frac{\mysum {n_i}}{2d} = D  \pars{2^{d + 1} - 1}^\frac{n}{2d} ,\] 
    что и требовалось.
\end{proof}

Таким образом, мы показали, что оценка в гипотезе~\ref{Hyp} достигается только на графах Алона, а для всех остальных графов количество независимых множеств отличается, как минимум в константу (зависящую от $d$) раз. Следущим шагом может стать изучение, на каких графах достигается максимум числа незвисимых множеств при $n$ не кратных $2d$, а также в графах, не содержащих $K_{d, d}$ как компонент связности.

\end{document}